\documentclass{article}
\usepackage{latexsym,amsfonts,palatino,eulervm,amsthm,graphicx,verbatim, float}
\usepackage[margin=2cm]{geometry}
\usepackage{amsmath}
\usepackage{algorithm}
\usepackage[noend]{algpseudocode}
\usepackage[affil-it]{authblk}

\makeatletter
\def\BState{\State\hskip-\ALG@thistlm}
\makeatother

\algnewcommand\And{\textbf{and} }

\newtheorem{theorem}{Theorem}[subsection]
\newtheorem{lemma}[theorem]{Lemma}
\newtheorem{corollary}[theorem]{Corollary}

\newtheorem{proposition}[theorem]{Proposition}

\theoremstyle{definition}
\newtheorem{definition}[theorem]{Definition}
\newtheorem{example}[theorem]{Example}

\def\n{\noindent}

\title{Alternating Signed Bipartite Graphs and Difference-1 Colourings}

\author{Cian O'Brien$^1$}
\author{Kevin Jennings}
\author{Rachel Quinlan}
\affil{National University of Ireland, Galway}

\begin{document}
\setlength{\parindent}{0pt}
\maketitle
\hspace{1cm}\footnotesize $^1$ Funded by Hardiman Research Scholarship \normalsize
\begin{abstract}
\noindent We investigate a class of 2-edge coloured bipartite graphs known as \emph{alternating signed bipartite graphs (ASBGs)} that encode the information in \emph{alternating sign matrices}. The central question is when a given bipartite graph admits an ASBG-colouring; a 2-edge colouring such that the resulting graph is an ASBG. We introduce the concept of a difference-1 colouring, a relaxation of the concept of an ASBG-colouring, and present a set of necessary and sufficient conditions for when a graph admits a difference-1 colouring. The relationship between distinct difference-1 colourings of a particular graph is characterised, and some classes of graphs for which all difference-1 colourings are ASBG-colourings are identified. One key step is Theorem \ref{multimatching}, which generalises Hall's Matching Theorem by describing a necessary and sufficient condition for the existence of a subgraph $H$ of a bipartite graph in which each vertex $v$ of $H$ has some prescribed degree $r(v)$.
\end{abstract}

\section{Introduction}

In this article, we develop a theme introduced by Brualdi, Kiernan, Meyer, and Schroeder in \cite{brualdibib}, by investigating a class of bipartite graphs related to \emph{alternating sign matrices}. In general, we may associate to any matrix the bipartite graph whose vertices correspond to rows and columns, and where an edge between the vertices representing Row $i$ and Column $j$ encodes the information that the $(i,j)$ entry is non-zero. Additional features of the non-zero entries (for example, sign) might be indicated by assigning colours to the edges. In the case of alternating sign matrices, the special matrix structure translates to particular combinatorial properties of the resulting bipartite graphs. We introduce the main objects of interest in this opening section.

\subsection{Alternating Sign Matrices and Alternating Signed Bipartite Graphs}

\begin{definition}
An \emph{alternating sign matrix (ASM)} is a ${(0,1,-1)}$-matrix in which all row and column sums are $1$, and the non-zero elements in each row and column alternate in sign.
\end{definition}

It is easily observed that permutation matrices are examples of ASMs, and there are contexts in which the concept of an ASM arises as a natural extension of a permutation. Alternating sign matrices were first investigated by Mills, Robbins, and Rumsey \cite{asmconjecturebib}, who observed their connection to a variant of the ordinary determinant function related to the technique of \emph{Dodgson condensation} \cite{dodgsonbib}. In their construction, the role of ASMs is similar to that of permutations in the usual definition of the determinant. This motivated the problem of enumerating the ASMs of size $n \times n$, leading to the \emph{Alternating Sign Matrix Conjecture}, namely that this number is
\[\frac{1!4!7! \ldots (3n-2)!}{n!(n+1)! \ldots (2n-1)!}\text{.}\]
Independent and very different proofs were published in 1996 by Zeilberger \cite{asmproof1bib} and Kuperberg \cite{asmproof2bib}, respectively using techniques from enumerative combinatorics and from statistical mechanics. Zeilberger's article establishes that the number of $n \times n$ ASMs is equal to the number of \emph{totally symmetric self-complementary plane partitions} in a $2n \times 2n \times 2n$ box, and the connection between these two classes of objects is further explored by Doran \cite{tsscppbib}. The connection to physics arises from the \emph{square ice} model for two-dimensional crystal structures; square patches of which correspond exactly to alternating sign matrices. Interest in alternating sign matrices intensified following the discovery of these deep connections between apparently disparate fields. A detailed and engaging account of the resolution of the alternating sign matrix conjecture can be found in the book by Bressoud \cite{bressoudbib}.

\*

Recent developments in the study of ASMs include an investigation of their spectral properties in \cite{spectralbib}, an extension of the concept of Latin square arising from the replacement of permutation matrices by ASMs in \cite{latinsquarebib}, and a study of some graphs arising from ASMs in \cite{brualdibib}. This latter article introduces the concept of the \emph{alternating signed bipartite graph (ASBG)} of an ASM, which is constructed as follows. The graph has a vertex for each row and each column of the matrix. Edges in the graph represent non-zero entries; the vertices corresponding to Row $i$ and Column $j$ are adjacent if and only if the entry in the $(i,j)$-position of the matrix is non-zero. Edges are marked as positive or negative, respectively represented by blue and red edge colours, according to the sign of the corresponding matrix entry.

\begin{example}\label{3x3_ASBGs}
All seven $3 \times 3$ ASMs and their corresponding ASBGs (up to isomorphism) are shown below:

\includegraphics[width = \textwidth]{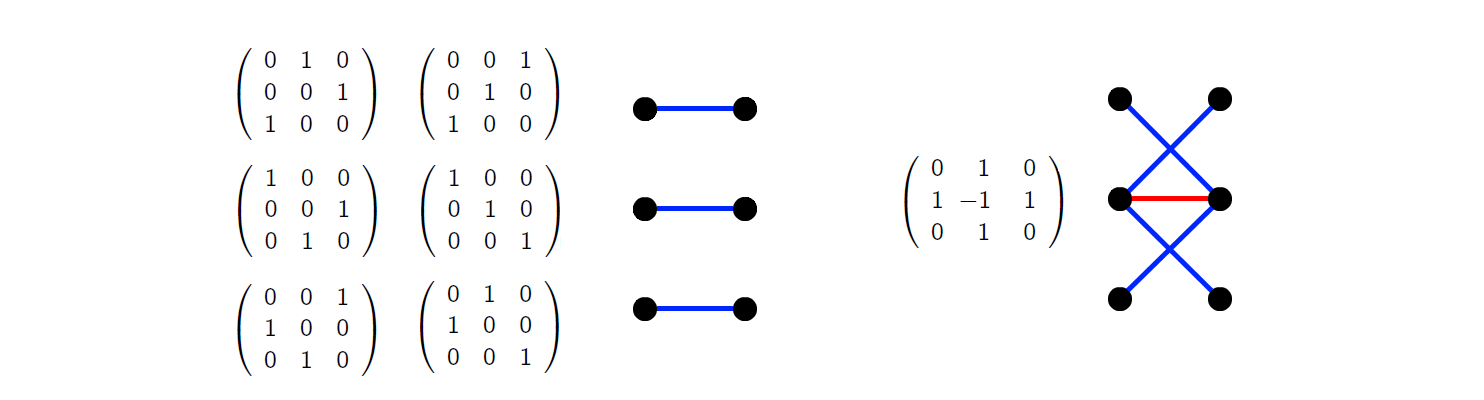}

\n\textsc{Note:} ASMs $A$ and $B$ correspond to the same ASBG if $B = PAQ$ for permutation matrices $P$ and $Q$.
\end{example}

From the viewpoint of bipartite graphs, we propose the following abstraction of the previous definition.
 
\begin{definition} \label{asbg_def} An \emph{alternating signed bipartite graph (ASBG)} is a bipartite graph ${G}$ with no isolated vertices and edges coloured blue and red, for which there exists an ordering of the vertices of ${G}$ such that, for each vertex $u$ of $G$ with neighbours ordered $v_1, v_2, \dots, v_k$, the edges $uv_1, uv_2, \dots, uv_k$ alternate in colour, starting and ending with blue.\end{definition}

Suppose that $G$ is an ASBG in the sense of Definition \ref{asbg_def}, with a bipartition $(P_1,P_2)$ of its vertex set. Let $u_1,\dots ,u_k$ be the vertices of $P_1$ and let $v_1,\dots ,v_l$ be the vertices of $P_2$. It is easily confirmed that the following assignment determines a  $k\times l$ alternating sign matrix $A(G)$.
$$
A(G)_{ij} = \left\{\begin{array}{rl} 1 & \mathrm{if\ }u_iv_j\ \mathrm{is\ a\ blue\ edge\ in}\ G \\
-1 & \mathrm{if\ }u_iv_j\ \mathrm{is\ a\ red\ edge\ in}\ G \\
0 & \mathrm{if\ }u_iv_j\ \mathrm{is \ not \ an\ edge\ in}\ G \\
\end{array}\right.
$$
Since every row and every column of $A(G)$ has entries summing to 1, the sum of all entries of $A(G)$ is equal both to $k$ and $l$. Thus $k=l$ and $G$ is balanced.

\*

Our main theme in this article is the problem of determining whether a given graph admits a 2-edge-colouring with respect to which it is an ASBG. The article is organised as follows. In Section 2, we introduce the concept of an ASBG-colouring, and a relaxation of this, which we refer to as a difference-1 colouring. We show that these two are equivalent in the case of a graph in which no edge belongs to multiple cycles. In Section 3, we establish general criteria for a graph to admit a difference-1 colouring, establishing a generalisation of Hall's Matching Theorem in the process. In the final section, we consider when a bipartite graph may admit multiple difference-1 colourings, and introduce some variants of such colourings.
\newpage
\section{ASBG-Colourings}

In this section, we consider properties of edge-colourings compatible with an ASBG structure.

\subsection{Obstacles to ASBG-Colourability}

We define a \emph{colouring} ${c}$ of a graph ${G}$ to be a function ${c: E(G) \rightarrow \{r,b\}}$, and denote the graph ${G}$ endowed with the colouring ${c}$ by ${G^c}$. If ${H}$ is a subgraph of ${G}$, then ${c}$ restricts to a colouring ${c_H}$ of ${H}$. We denote by ${H^c}$ the graph ${H}$ with edges coloured according to ${c_H}$.

If $G^c$ is an ASBG and $v$ is a vertex of $G^c$, then the numbers ${deg^B(v)}$ and ${deg^R(v)}$ of blue and red edges incident with ${v}$, respectively, must satisfy the following relation for all vertices $v$:
\begin{equation}\label{degB-degR} deg^B(v) - deg^R(v) = 1 \text{.} \end{equation}

\begin{definition} A 2-edge colouring $c$ of a graph $G$ is an \emph{ASBG-colouring} of $G$ if $G^c$ is an ASBG. The graph $G$ is {ASBG-colourable} if there exists an ASBG-colouring of $G$.\end{definition}

The following are some basic necessary (but not sufficient) criteria that a graph ${G}$ must meet for it to be ASBG-colourable:
\begin{itemize}
\item ${G}$ must be bipartite and balanced.
\item The degree of each vertex in ${G}$ must be odd, from \eqref{degB-degR}.
\end{itemize}

\begin{example}\label{unfeasible}
Let $G$ be the following graph:

\includegraphics[width = \textwidth]{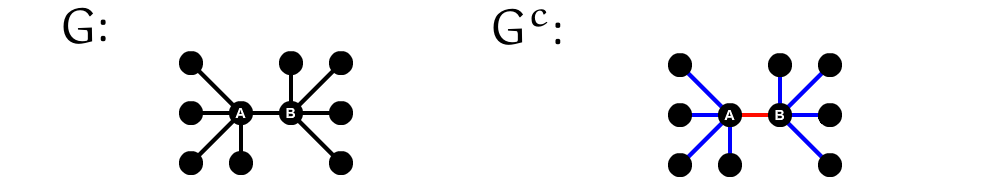}

We attempt to colour the edges of $G$ with a colouring $c$, so that \eqref{degB-degR} is satisfied at every vertex:
\begin{itemize}
\item All edges incident with vertices of degree ${1}$ must be blue.
\item Now $4$ out of $5$ edges incident with vertices $A$ and $B$ are blue, so it is not possible for $G^c$ to satisfy \eqref{degB-degR} at each vertex.
\end{itemize}

Therefore ${G}$ is not ASBG-colourable.
\end{example}

\begin{example}\label{unconfigurable}
Now consider the following graph ${G}$ with edge-colouring ${c}$.

\includegraphics[width = \textwidth]{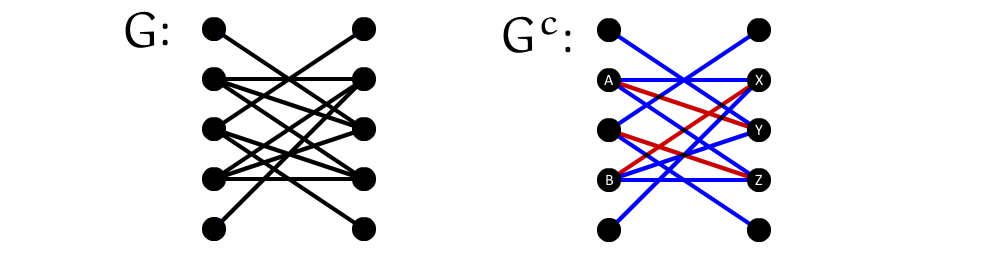}

We observer that $G$ has a unique colouring that satisfies \eqref{degB-degR} at every vertex. However, we claim that there is no ordering of the vertices that satisfies the condition of Definition \ref{asbg_def} for this colouring. The conditions of Definition \ref{asbg_def}, applied to the vertex $A$, requires that $Y$ occurs between $X$ and $Z$ in an ordering of the vertices of $G$. The same condition applied to $B$ requires that $X$ occurs between $Y$ and $Z$. Since these requirements are incompatible, we conclude that $G$ is not ASBG-colourable.
\end{example}

Examples \ref{unfeasible} and \ref{unconfigurable} demonstrate two different ways in which a graph can fail to be ASBG-colourable. If a graph ${G}$ has an edge colouring ${c}$ that satisfies \eqref{degB-degR} at every vertex, we say that ${G}$ has a \emph{difference-${1}$ colouring}, and if the vertices of ${G^c}$ can be ordered such that the edges incident with each vertex alternate in colour, beginning and ending with blue, we say that ${G^c}$ is \emph{configurable}.

\*

In this paper, we will present necessary and sufficient conditions for a given graph to have a difference-1 colouring, as well as classes of graphs which are configurable for any difference-1 colouring. We will also give a generalisation of \emph{Hall's Matching Theorem}, which will be needed further in the paper.


\subsection{Difference-$1$ Colourings}

\begin{definition}\label{dif-1_def} A 2-edge colouring $c$ of a graph ${G}$ is a \emph{difference-$1$ colouring} of $G$ if $G^c$ satisfies $deg^B(v)-deg^R(v) = 1$ at every vertex $v$.\end{definition}

\n\textsc{Note:} We have introduced the concept of a difference-1 colouring specifically to address the question of ASBG-colourability. Therefore, we will consider the existence of difference-1 colourings only for balanced bipartite graphs.

\begin{definition}\label{config_def} A difference-1 colouring ${c}$ of a graph $G$ is \emph{configurable} if $G^c$ satisfies the conditions of Definition \ref{asbg_def}.\end{definition}

\n\textsc{Note:} A graph has a difference-1 colouring if and only if each of its connected components has a difference-1 colouring. Similarly, a graph is an ASBG if and only if each of its connected components is. Therefore, we will consider the existence of difference-1 and ASBG-colourings only for connected graphs, and so all graphs from now on should be assumed to be connected, balanced, and bipartite, unless otherwise stated.

\*

We remark that the property of configurability can be usefully visualised in terms of embedding the vertices in the plane along two parallel lines so that the edges incident with each vertex alternate in colour, as in the diagrams in Examples \ref{3x3_ASBGs} and \ref{unconfigurable}.

\begin{lemma} Let ${G}$ be a bipartite graph with a difference-1 colouring. Then ${G}$ is balanced. \end{lemma}
\begin{proof}
Let $c$ be a difference-1 colouring of $G$, and let ${b}$ and ${r}$ be the number of blue and red edges in ${G^c}$, respectively. Let ${(P_1, P_2)}$ be the bipartition of ${V(G)}$. Each vertex ${v}$ of ${P_1}$ satisfies ${deg^B(v) - deg^R(v) = 1}$. Summing this expression over all ${v \in P_1}$, we have ${|P_1| = b - r}$. Similarly, ${|P_2| = b - r}$. Therefore ${|P_1| = |P_2|}$, so ${G}$ is balanced.
\end{proof}

\begin{definition} A \emph{leaf} is a vertex of degree ${1}$.\end{definition}
\begin{definition} A \emph{twig} is a configuration of three vertices, consisting of two leaves incident with the third vertex of the twig (called the base of the twig), which has degree ${3}$. \end{definition}

One useful property of leaves and twigs is that the colouring of their edges in any difference-1 colouring is uniquely determined. All edges incident with leaves must be coloured blue, which means that the remaining edge incident with the base of a twig must be coloured red, to satisfy ${(\ref{degB-degR})}$.

\begin{definition} A \emph{leaf-twig configuration} at a vertex ${v}$ is a configuration of four vertices (distinct from $v$), consisting of a leaf and a twig, where the base of the twig and the leaf are both incident with ${v}$.\end{definition}

\n\textsc{Note:} We refer to the operation of deleting the four vertices of a leaf-twig configuration, and their four incident edges, as \emph{removing a leaf-twig configuration from G}. We also refer to the operation of adding four vertices to ${G}$ in such a configuration, with the leaf and the base incident with a vertex ${v}$ of ${G}$, as \emph{adding a leaf-twig configuration to ${G}$ at ${v}$}.

\includegraphics[width = \textwidth]{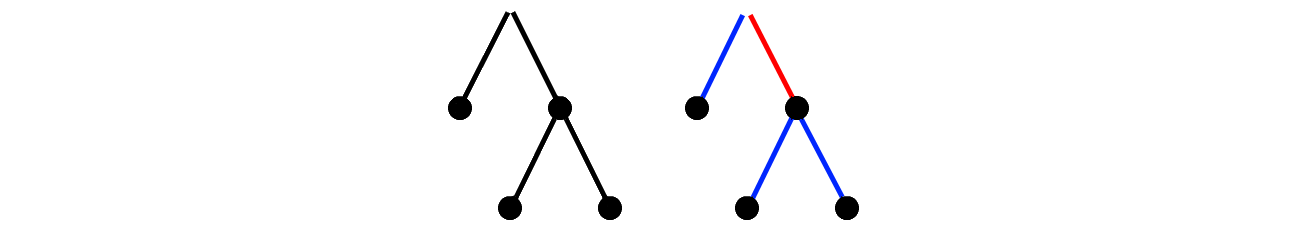}
\begin{center}
{\footnotesize A leaf-twig configuration with its only difference-1 colouring.}
\end{center}

\begin{lemma}\label{leaf-twig_removal_lemma} Let $G$ and $G'$ be graphs with the property that $G'$ is obtained from $G$ by the addition of a leaf-twig configuration. Then $G'$ has a difference-1 colouring if and only if $G$ has a difference-1 colouring.
\end{lemma}
\begin{proof}
Since the addition of a leaf-twig configuration involves one additional edge of each colour at a single vertex of $G$, it is easily observed that any difference-$1$ colouring of $G$ extends in a unique way to a difference-$1$ colouring of $G'$. On the other hand, any difference-$1$ colouring of $G'$ restricts to a difference-$1$ colouring of $G$.
\end{proof}

A consequence of Lemma \ref{leaf-twig_removal_lemma} is that repeated addition and/or removal of leaf-twig configurations does not affect the status of a graph with respect to the existence of a difference-$1$ colouring.

\*

\n\textsc{Note:} For any graph with a difference-1 colouring that is not configurable, it is possible to add leaf-twig configurations so the resulting graph is configurable. Adding leaf-twig configurations to a configurable graph will always result in a configurable graph. 

\begin{definition} Given a graph ${G}$, we define the \emph{reduced form} of ${G}$ be the graph that results from deleting leaf-twig configurations from ${G}$ until none remain. We also refer to a graph as \emph{reduced} if it possesses no leaf-twig configuration. \end{definition}

\n\textsc{Note:} Although the reduced form of $G$ is not necessarily uniquely determined as a subgraph of $G$, it is uniquely determined as a graph, up to isomorphism.

\*

It is a consequence of Lemma \ref{leaf-twig_removal_lemma} that ${G}$ has a difference-1 colouring if and only if its reduced form has a difference-1 colouring.

\subsection{Configurability for Cactus Graphs}\label{cactus_section}

In this section, we examine a class of graphs for which any graph with a difference-1 colouring is configurable; namely \emph{cactus graphs}.

\begin{definition} A \emph{cactus graph} is a connected graph in which any pair of cycles share at most one vertex. \end{definition}

Equivalently, a cactus graph is a connected graph in which any edge belongs to at most one cycle. Our main result in this section is Theorem \ref{ced_difference-1_is_asbg}, which asserts that for a cactus graph, the existence of a difference-1 colouring is a sufficient condition for ASBG-colourability. Lemma \ref{cactus} provides the key technical ingredient in the proof of our main theorem in this section. Our proof of this statement makes use of a partition of the vertex set of a bipartite cactus graph, which is explained in the following lemma. We denote the minimum distance in a graph between vertices $v$ and $u$ by $d(v,u)$.

\begin{lemma}\label{cactus}
Let $v$ be a vertex of a bipartite cactus graph $G$. Write $k = \max_{u\in V(G)}d(v,u)$, and for $i=0,\dots ,k$ define
$$
V_i = \{u\in V(G): d(v,u)=i\}.
$$
\begin{enumerate}
\item
If $1\le i\le k$ and $x\in V_i$, then $x$ has at most two neighbours in $V_{i-1}$. 
\item
If $1\le i\le k-1$ and $x\in V_i$, then there is at most one vertex $y$ in $V_i$, distinct from $x$, for which $x$ and $y$ have a common neighbour in $V_{i+1}$. Moreover, if such a vertex $y$ exists, then $x$ has only one neighbour in $V_{i-1}$. 
\end{enumerate}
\end{lemma}

\begin{proof}
We note that $V(G)$ is the disjoint union of the sets $V_i$, and that, since $G$ is bipartite, the neighbours of any vertex in $V_i$ belong to $V_{i-1}\cup V_{i+1}$, for $1\le i< k$. 

For item 1, suppose that $x$ has distinct neighbours $u_1$, $u_2$, $u_3$ in $V_{i-1}$. Each of the edges $xu_1,\ xu_2,\ xu_3$ is the intial edge of a path in $G$ from $x$ to $v$. It follows that  $xu_2$ belongs to two distinct cycles, one including the edge $xu_1$ and one including the edge $xu_3$, contrary to the hypothesis that $G$ is a cactus graph.

For item 2, suppose that $y_1$ and $y_2$ are two vertices of $V_i$, distinct from $x$, with the property that $z_1$ is a common neighbour of $y_1$ and $x$ in $V_{i+1}$, and $z_2$ is a common neighbour of $y_2$ and $x$ in $V_{i+1}$. Let $u$ be a neighbour of $x$ in $V_{i-1}$. Then the edge $xu$ belongs to two distinct cycles in $G$, one including the edges $z_1x$ and $z_1y_1$, and the other including the edges $z_2x$ and $z_2y_2$. 

Thus at most one element of $V_i\backslash\{x\}$ shares a neighbour $z$ with $x$ in $V_{i+1}$. Suppose that $y$ is such a vertex, and that $x$ has two neighbours $u_1$ and $u_2$ in $V_{i-1}$. Then the edge $u_1x$ belongs to a cycle in $G$ that includes the edges $xz$ and $zy$, and also to a cycle in $G$ that includes the edge $xu_2$ and no vertex of $V_{i+1}$. From this contradiction we conclude that if $x$ shares a neighbour in $V_{i+1}$ with another vertex of $V_i$, then $x$ has only one neighbour in $V_{i-1}$. 
\end{proof}

\begin{theorem}\label{ced_difference-1_is_asbg} Let ${G}$ be a cactus graph with difference-1 colouring ${c}$. Then ${G^c}$ is configurable. \end{theorem}

\begin{proof}We choose a vertex $v$ of $G$ and partition the vertex set of $G$ into non-empty subsets $V_0=\{v\}, V_1,\dots ,V_k$ as in the statement of Lemma \ref{cactus}. We note that $(P_1,P_2)$ is a bipartition of the vertex set of $G$, where $P_1=\cup_{i \text{ even}}V_i$ and $P_2=\cup_{i \text{ odd}}V_i$. 
We let $L$ and $M$ be parallel lines embedded on the plane and position the vertex $v$ on the line $L$. We configure $G$ by positioning the vertices of $P_1$ and $P_2$ at distinct locations on $L$ and $M$ respectively, so that the order in which the vertices are positioned along the two lines satisfies the requirements of Definition \ref{config_def}. For $1\le i\le k$, the $i$th step in the process involves the positioning of the vertices of $V_i$ on $L$ or $M$, according to whether $i$ is even or odd. 

Suppose that $r$ steps have been completed and that every vertex of $\cup_{i=0}^{r-1}V_i$ has the property that its neighbours are positioned in a manner consistent with an ASBG-colouring. Let $x_1,\dots ,x_t$ be the vertices of $V_r$. We position the vertices of $V_{r+1}$ (on $L$ or $M$) in $t$ steps, the first of which is to position the neighbours of $x_1$. At most two neighbours of $x_1$ (in $V_{r-1}$) are already positioned; all further neighbours of $x_1$ belong to $V_{r+1}$ and may be positioned in a manner that satisfies the alternating condition on colours of edges incident with $x_1$. At step $j$ we position the neighbours of $x_j$ in $V_{r+1}$. From Lemma \ref{cactus} it follows that at most two neighbours of $x_j$ in $G$ have already been assigned positions at this stage, potentially two from $V_{r-1}$ or one each from $V_{r-1}$ and $V_{r+1}$. In any case we may position the remaining neighbours of $x_j$ to satisfy the alternating condition on coloured edges. 

This iterative process results in a configuration of $G$. 
\end{proof}

\begin{example} \label{ced_config}
The diagram below demonstrates how the construction of Theorem \ref{ced_difference-1_is_asbg} might apply to a particular cactus graph with a difference-1 colouring, for a choice of initial vertex $v$.

\includegraphics[width = \textwidth]{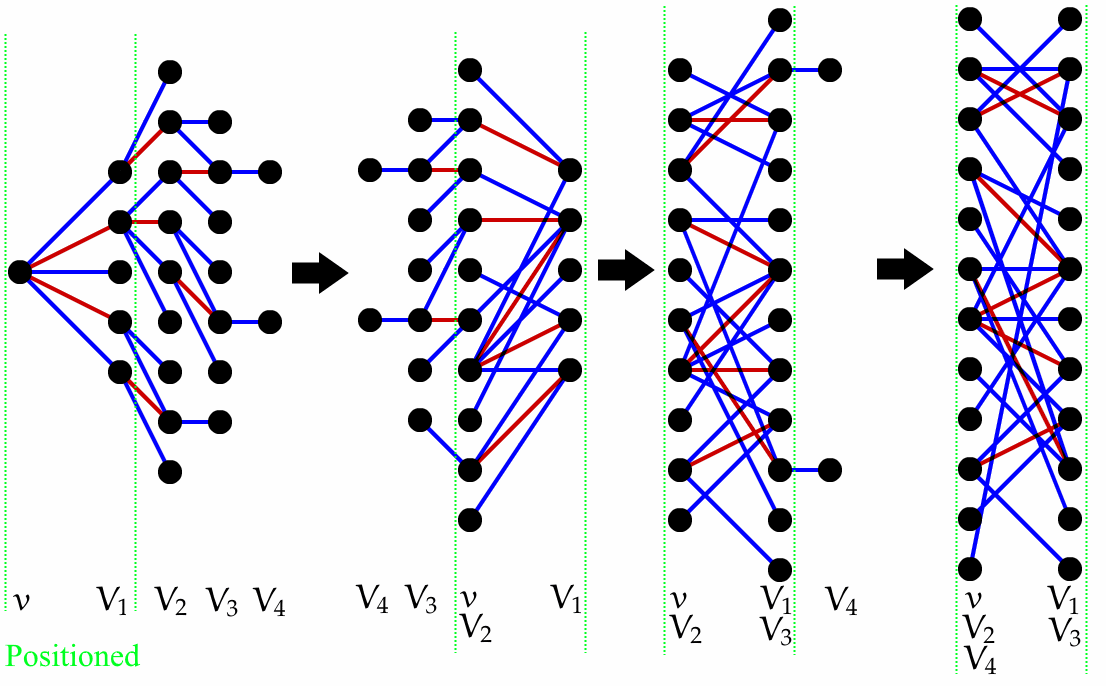}

\end{example}


\section{Deciding the Existence of a Difference-1 Colouring}

In this section, we develop methods of determining whether a given bipartite graph admits a difference-1 colouring. We consider this question first for trees and unicyclic graphs, where the analysis is considerably easier.

\subsection{Trees}

We now consider necessary and sufficient conditions for a tree to have a difference-1 colouring. As we have seen in Section \ref{cactus_section}, this will resolve the question of ASBG-colourability for trees, as every difference-1 colouring of a tree is configurable. We first note that the only connected ASBG that contains no red edges is $P_2$, the path graph on two vertices (with its  edge coloured blue).

\begin{theorem}\label{asbt-colouring} Let ${T}$ be a tree. Then ${T}$ has a difference-1 colouring if and only if its reduced form is $P_2$. \end{theorem}
\begin{proof}
Suppose the reduced form of ${T}$ is $P_2$. Then $T$ can be constructed by repeatedly adding leaf-twig configurations to $P_2$. Hence, by Lemma \ref{leaf-twig_removal_lemma}, $T$ has a difference-1 colouring.

\*

Now suppose that $T$ has a difference-1 colouring $c$ and is not $P_2$. Choose any vertex ${r}$ of ${T}$. Let ${u}$ be a leaf of ${T}$ which is furthest from ${r}$, and let $b$ be the neighbour of $u$ on the unique path from $u$ to $r$. As $T$ has a difference-1 colouring, every vertex of $T$ has odd degree. Therefore $u$ shares its neigbour $b$ with at least one other vertex $u_1$ that is further from $r$ than $b$. As $u$ is a furthest leaf from $r$, this means that $u_1$ is also a leaf. As $u$ and $u_1$ are both leaves, this means the edges $ub$ and $u_1b$ must be blue in $T^c$. This means that a third edge incident with $b$ and another vertex $v$ must be red. If $v$ is not on the unique path from $b$ to $r$, then $v$ has a neighbour which is a leaf that is further from $r$ than $u$, which is contrary to the choice of $u$. Therefore $v$ is on the unique path from $b$ to $r$. As $c$ is a difference-1 colouring and $bv$ is red in $T^c$, $v$ must be incident with at least two blue edges. Let $l$ be a neighbour of $v$ which is not on the unique path from $v$ to $r$, such that $lv$ is blue. As $d(r,u) = d(r,l) + 1$, this means that $l$ must be a leaf, otherwise it would be incident with at least one red edge leading to a leaf further from $r$ than $u$. This means that ${u}$, ${u_1}$, ${b}$, and ${l}$ make a leaf-twig configuration, with ${b}$ being the base of the twig, and ${l}$ the leaf. Removing this results in another tree with a difference-1 colouring (Lemma \ref{leaf-twig_removal_lemma}). Inductively, this means that leaf-twig configurations can be removed until there are no more red edges. As $P_2$ is the only connected ASBG with no red edges, this process will reduce ${T}$ to $P_2$.
\end{proof}

\begin{corollary} A tree has an ASBG-colouring if and only if its reduced form is $P_2$, and this colouring is unique. \end{corollary}

Note that the process outlined in the proof of Theorem \ref{asbt-colouring} not only gives an existence condition for an ASBG-colouring of a tree, it gives the colouring. The following observation will be needed later, for the proof of Lemma \ref{local-tree-lemma}.

\begin{corollary} \label{asbt-corollary} Let ${T}$ be an ASBG-colourable tree and let ${r}$ be a vertex of ${T}$. Then there is a sequence of leaf-twig removals that reduces ${T}$ to a subgraph that consists of only leaves and twigs attached to $r$, with one more leaf than twig. \end{corollary}
\begin{proof}
In the proof of Theorem \ref{asbt-colouring}, it was shown that a furthest vertex from ${r}$ is always part of a removable leaf-twig configuration. So we succesively remove the furthest leaf-twig configuration from $r$ until any remaining leaf-twig configurations are attached to $r$. We call the resulting tree $T'$. Theorem \ref{asbt-colouring} tells us that it is possible to remove leaf-twig configurations from $T$ until only $P_2$ remains, where $r$ is one of the vertices of this $P_2$. Therefore $T'$ must be $P_2$ with extra leaf-twig configurations at $r$, which means that $T'$ consists only of leaves and twigs attached to $r$, with one more leaf than twig.
\end{proof}


\subsection{Unicyclic Graphs}

\begin{definition} A graph is \emph{unicyclic} if it is connected and contains exactly one cycle. \end{definition}

When analysing the ASBG-colourability of graphs which are not trees, it is useful to define the \emph{skeleton} of a graph.

\begin{definition} For a graph ${G}$ (containing at least one cycle), we refer to the subgraph that results from repeatedly removing leaves and their incident edges, until none remain, as \emph{the skeleton} ${Sk(G)}$ of $G$. \end{definition}

\n\textsc{Note:} ${Sk(G)}$ is determined as a subgraph of ${G}$. The process of repeatedly deleting leaves from a tree terminates with a single vertex, but which vertex remains depends on the order in which leaves are deleted. For this reason, we do not define the skeleton of a tree.

\*

We also note that if ${G'}$ is the reduced form of ${G}$, then ${Sk(G') = Sk(G)}$.

\begin{definition} A \emph{junction} in a graph ${G}$ is a vertex $v$ with ${deg_{Sk(G)}(v) \geq 3}$. \end{definition}

\n\textsc{Note:} If ${G}$ is unicyclic, then the skeleton ${Sk(G)}$ is the graph consisting of the cycle in ${G}$, and ${G}$ has no junctions.

\begin{definition} For any graph $G$ that contains a cycle, let ${v}$ be a vertex in ${Sk(G)}$. We define the \emph{local tree at v}, ${T_v}$, to be the connected component containing ${v}$ that remains when all edges of ${Sk(G)}$ are deleted from ${G}$. \end{definition}

\begin{example}
${\;}$

\includegraphics[width = \linewidth]{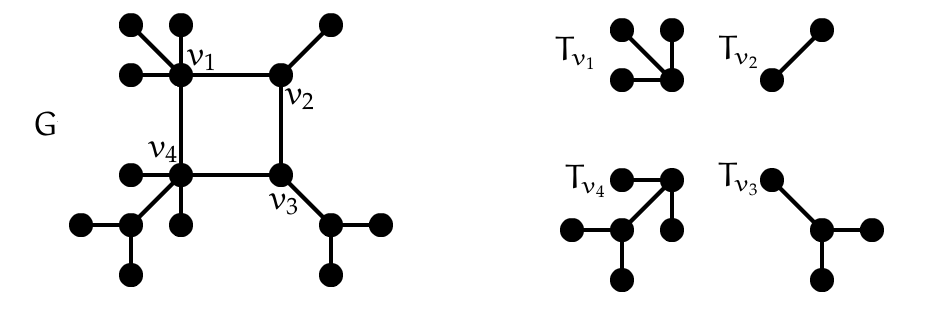}
\end{example}

\begin{lemma}\label{local-tree-lemma} Let ${G}$ be a graph with a difference-1 colouring ${c}$ and reduced form $H$. Then for each vertex ${v \in Sk(H)}$, the local tree ${T_v}$ consists either only of ${v}$, or of $v$ with only leaves or only twigs incident with it. \end{lemma}

\begin{proof} The number of blue edges at ${v}$ in ${H^c}$ is one greater than the number of red edges at ${v}$ in ${H^c}$. This also is the case for every vertex in ${T_v^c}$ except possibly for ${v}$ itself.

Let ${d = deg_{T_v}^{\hspace{2mm}B}(v) - deg_{T_v}^{\hspace{2mm}R}(v)}$. If ${d > 1}$, we attach ${d-1}$ twigs to ${v}$. If ${d < 1}$, we attach ${1-d}$ leaves to ${v}$. We call the resulting coloured tree ${T_v'}$.

${T_v'}$ is an ASBG, which means that leaf-twig configurations can be removed from $T_v'$ until only leaf-twig configurations attached to $v$ remain (Corollary \ref{asbt-corollary}). If there are any leaf-twig configurations attached to $v$ that are also in $T_v$, we remove them. We have now removed all leaf-twig configurations in $T_v'$ that are also in $G$, and the remaining graph consists of a copy of $P_2$ containing $v$ with $|d-1|$ leaf-twig configurations attached to $v$. We can now delete the $|d-1|$ twigs or leaves that we attached to $v$ to make $T_v'$ that were not part of $T_v$. We are now left with the reduced form of $T_v$; a subgraph of ${T_v}$ that consists either only of $v$, or of $v$ with only leaves or only twigs incident with it.
\end{proof}

\begin{example}
${\;}$

\includegraphics[width = \linewidth]{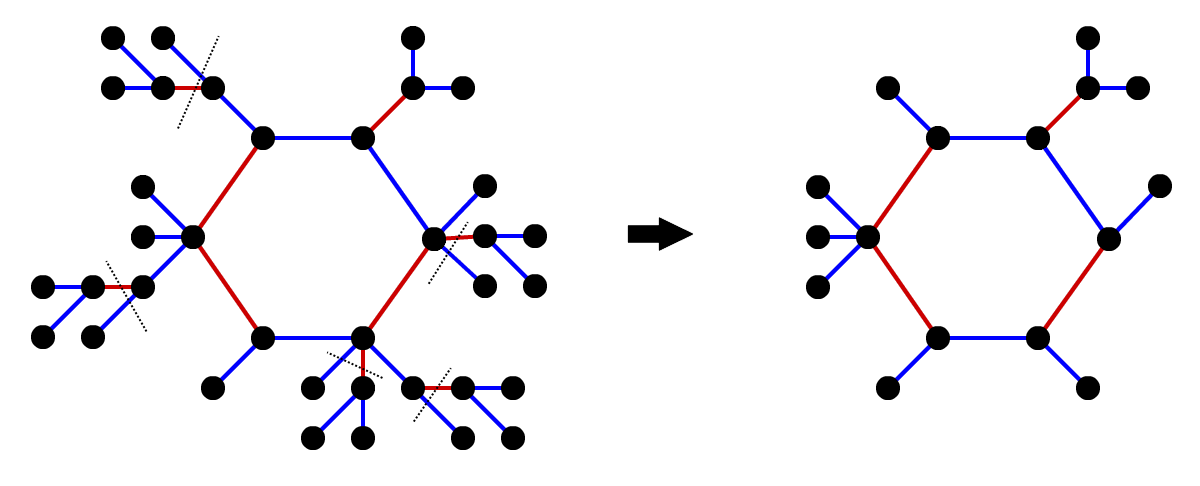}
\end{example}

\begin{corollary}\label{ltt-corollary}
Let ${G}$ be a graph and ${H}$ be the reduced form of ${G}$. Then if $G$ has a difference-1 colouring $c$, any vertex of degree ${2}$ in ${Sk(G)}$ is of one of the following types:
\begin{itemize}
\item \textbf{Leaf-Type:} A vertex of degree $3$ in $H$, with a leaf incident with it. If a blue and red edge meet at a vertex of degree $2$ in $Sk(G^c)$, then $v$ must be a leaf-type vertex in $G$.
\item \textbf{Twig-Type:} A vertex of degree ${3}$ in ${H}$, with the base of a twig incident with it. If two blue edges meet at a vertex of degree $2$ in $Sk(G^c)$, then $v$ must be a twig-type vertex in $G$.
\item \textbf{Triple-Type:} A vertex of degree ${5}$ in ${H}$, with three leaves incident with it. If two red edges meet at a vertex of degree $2$ in $Sk(G^c)$, then $v$ must be a triple-type vertex in $G$.
\end{itemize}
\end{corollary}

\n\textsc{Note:} We only use the terms leaf-type, twig-type, and triple-type to classify vertices of degree $2$ in $Sk(G)$. Vertices of higher degree in $Sk(G)$ will be discussed later. We consider twig-type and triple-type vertices to be of \emph{opposite type} to one another.

\begin{definition} Let $G$ be a graph whose skeleton contains at least one vertex that is not leaf-type. A \emph{limb} of ${G}$ is a subgraph ${H}$ of ${Sk(G)}$ such that the edges of ${H}$ form a trail whose only non-leaf-type vertices are its (not necessarily distinct) endpoints. \end{definition}

We observe that every edge of ${Sk(G)}$ belongs to exactly one limb, and that the skeleton of a graph is the edge-disjoint union of its limbs.

\begin{lemma} \label{odd-even-distance-lemma} Let ${G}$ be a graph with a difference-1 colouring and let  ${P = v_1 \dots v_k}$ be a limb in ${Sk(G)}$. Then ${P}$ has odd length if ${v_1}$ and ${v_k}$ are both of twig-type or triple-type, and even length if $v_1$ and $v_k$ are of opposite type. \end{lemma}

\begin{proof}
Let $v_1$ and $v_k$ be vertices of the same or opposite type to one another. This means that each is a twig-type or triple-type vertex. As $v_1$ is either a twig-type or triple-type vertex, both edges incident with $v_1$ are the same colour. The same is true for $v_k$. As ${v_2, \dots, v_{k-1}}$ are all leaf-type vertices, they are each incident with an edge of each colour, which means that edge colours alternate along $P$. Therefore, if $k-1$ is odd, the colours of the edges incident with $v_1$ are the same as those incident with $v_k$, meaning that $v_1$ and $v_k$ are of the same vertex type. If $k-1$ is even, then the colours of the edges incident with $v_1$ are different to those incident with $v_k$, meaning that $v_1$ and $v_k$ are of opposite vertex type.
\end{proof}

\begin{proposition}\label{unicyclic_prop} Let ${G}$ be a bipartite unicyclic graph. Then ${G}$ is ASBG-colourable if and only if ${G}$ satisfies the following:
\begin{itemize}
\item Each vertex in Sk(G) is either a leaf-type, twig-type, or triple-type vertex;
\item Each limb has odd (even) length if its endpoints are the same (opposite) type.
\end{itemize}
\end{proposition}
\begin{proof}
Suppose ${G}$ satisfies the above conditions. We now give $G$ a colouring $c$, as follows. At any twig-type vertex $v$ of $Sk(G)$, we colour the edges of $Sk(G)$ incident with $v$ blue, we colour the edge between $v$ and the base of the twig red, and the other two edges of the twig blue. So $v$ is incident with one more blue edge than red. At any triple-type vertex $v$ of $Sk(G)$, we colour the edges of $Sk(G)$ incident with $v$ red, and we colour the edges between $v$ and the three leaves blue. So $v$ is incident with one more blue edge than red. As all vertices of $Sk(G)$ that are of the opposite type are an even distance apart, vertices of the opposite type cannot be neighbours, and so it is possible to colour $G$ in this way. Along each limb of $Sk(G)$, the first edge is now coloured red or blue according to the type of vertex it is incident with. If the other vertex incident with this edge is of leaf-type, we colour the next edge along the limb the opposite colour to the previous edge. We continue to do this until we reach the other end of the limb. As the limb has odd (even) length if the vertices are of the same (opposite) type, the colour of the last two edges of the limb are the opposite to one another, and the vertex incident with these two edges are of leaf-type. Finally, we colour all edges between all leaf-type vertices and leaves blue. Now all leaf-type vertices are inciden with one more blue edge than red, and $c$ is a difference-1 colouring of $G$. As ${G}$ is unicyclic, we know that ${G^c}$ is configurable (Theorem \ref{ced_difference-1_is_asbg}). So $G$ is ASBG-colourable.

\*

On the other hand, suppose ${G}$ has an ASBG-colouring ${c}$. From Corollary \ref{ltt-corollary} and  Lemma \ref{odd-even-distance-lemma}, we have that each vertex in $Sk(G)$ is of leaf, twig, or triple-type, and that each limb has odd (even) length if its endpoints are the same (opposite) type.
\end{proof}


\subsection{Graphs With Junctions}

We now turn our attention to bipartite graphs possessing junctions, where the problems of determining the existence of a difference-1 colouring and configurability are both considerably more complicated.

\*

We have identified necessary and sufficient conditions for trees and unicyclic graphs to admit difference-1 colourings. In this section, we extend our analysis to the situation of graphs whose skeletons include junctions. We note that Lemma \ref{local-tree-lemma}, Corollary \ref{ltt-corollary}, and Lemma \ref{odd-even-distance-lemma} provide a partial test for difference-1 colourability, in the sense that a bipartite graph whose reduced form fails to satisfy the conditions of these results does not have a difference-1 colouring. The main content of this section is an algorithm whose purpose is to detect obstacles to a graph that passes these conditions having a difference-1 colouring, or to confirm the existence of a difference-1 colouring. We may restrict our attention to reduced graphs satisfying these conditions.

\*

We now describe the operation of the algorithm on a reduced graph $G$ satisfying the conditions of Lemma \ref{local-tree-lemma}, Corollary \ref{ltt-corollary}, and Lemma \ref{odd-even-distance-lemma}. The algorithm initializes $J$ to be the empty set and then iterates the following step. Each iteration involves at least one addition of a new element to $J$.

A junction $j$ with $j \not \in J$ is chosen, and added to the set $J$. Integer weights are assigned in turn to each edge $e$ incident with $j$ as follows.
\begin{itemize}
\item If $e$ is not in $Sk(G)$, either $e$ is incident with a leaf and is assigned the weight $1$ (indicating that $e$ must be blue in any difference-1 colouring of $G$), or $e$ is incident with the base of a twig and is assigned the weight $-1$ (indicating red).
\item If $e$ is in $Sk(G)$, and the limb to which $e$ belongs ends in a vertex of twig/triple-type, then the colour of $e$ in any difference-$1$ colouring is determined by this limb, and $e$ is assigned the weight $1$ or $-1$ accordingly.
\item The remaining possibility is that the non-leaf type vertex on the other end of the limb starting at $j$ along $e$ is a junction $j_1$. In this case, the algorithm proceeds as follows:
\begin{itemize}
\item If $j_1 \in J$, then the weight 0 is assigned to $e$.
\item If $j_1 \not \in J$ (which means, in particular, that $j_1 \not = j$), then $j_1$ is adjoined to the set $J$, and the assignment procedure is applied recursively to all edges of $G$ incident with $j_1$, with the exception of the edge $e_1$ that belongs to the same limb as $e$. An integer weight is then assigned to $e_1$ by the requirement that the sum of weights on all edges incident with $j_1$ must be $1$. Then $x_e$, the weight assigned to $e$ is assigned by $x_e = \pm x_{e_1}$, according to whether the limb containing $e$ and $e_1$ requires that these two edges have the same or opposite colours.
\end{itemize}
\end{itemize}

This iteration of the algorithm concludes when weights have been assigned to all edges incident with junctions that can be reached from $j$ via limbs of $Sk(G)$ that do not include twig or triple-type vertices. No more iterations of the algorithm are run when all junctions of $G$ are included in $J$.

If the algorithm identifies a junction for which the sum of the incident weights is not $1$, then the algorithm returns \emph{False} to indicate that $G$ does not admit a difference-$1$ colouring. Otherwise, it returns \emph{True}, and the algorithm proceeds to assign weigths to all edges of the graph that are not incident with junctions, so that all edges of $G$ have now been assigned a weight. If all assigned weights are $1$, $-1$, or $0$, then $G$ admits a difference-$1$ colouring, which we can easily determine from the output of the algorithm. Otherwise, the algorithm has assigned some \emph{surplus weights}; weights which have magnitude greater than $1$. In this case, the existence or not of a difference-$1$ colouring depends on properties of the graph consisting of those limbs of $Sk(G)$ that include edges that have been assigned surplus or zero weights. This theme is developed in Section \ref{redistrib}. We now present the details of the algorithm in a pseudocode format.

\begin{algorithm}[H]
\caption{Difference-1 Colouring Algorithm}\label{junction_weight_algorithm}
\begin{algorithmic}[1]
\State $J \gets \{ \}$
\Function{assign}{j}
\For {e in j.edges} \Comment{j.edges is the set of edges incident with j}
\If {e.next $=$ "leaf"}\Comment{if $e \not\in Sk(G)$, e.next is the type of the other vertex incident with e}
\State e.weight $\gets 1$
\ElsIf {e.next $=$ "base\_of\_twig"}
\State e.weight $\gets -1$
\ElsIf {e.next $=$ "twig-type"} \Comment{e.next is type of last vertex of limb L from j along e}
\State e.weight $\gets (-1)^{\text{e.distance} + 1}$ \Comment{e.distance returns the length of L}
\ElsIf {e.next $=$ "triple-type"}
\State e.weight $\gets (-1)^{\text{e.distance}}$
\ElsIf {e.next $=$ "junction"}
\State $j_1 \gets$ e.next\_junction \Comment{We must first calculate the weights of edges incident with $j_1$}
\If{$j_1 \in J$}
\State e.weight $\gets 0$
\Else
\State $e_1 \gets$ e.last\_edge \Comment{Label the last edge on L}
\State J.append($j_1$)
\State $e_1$.next $\gets$ "null" \Comment{In the recursive step, the weight of $e_1$ should not be calculated}
\State \textproc{assign}($j_1$)
\State $e_1$.weight $\gets 1 -$ sum(f.weight \textbf{for} f \textbf{in} ($j_1$.edges $\setminus \{e1\}$))
\State e.weight $\gets (-1)^{\text{e.distance} - 1}e_1$.weight
\EndIf
\EndIf
\EndFor
\EndFunction

\While{$|J| < |\text{junctions}|$}
\State $junct \gets$ (junctions $\setminus J$) [0]
\State J.append(j)
\State \textproc{assign}(j)
\EndWhile

\Function{w}{j}
\Return sum(e.weight for e in j.edges)
\EndFunction

\State result $\gets$ True

\For{j in junctions}
\If {\textproc{w}(j) $\not = 1$}
\State result $\gets$ False \Comment{If w(j) is not 1 for every junction j, then the algorithm returns False}
\EndIf
\EndFor

\If{result}
\For{e in edges}
\If {e.weight = "null"} \Comment{If e is not incident with a junction, it has no weight}
\State e.weight $\gets (-1)^{e.\text{distance}}$e.next\_weight \Comment{We now assign e a weight}
\EndIf
\EndFor
\EndIf

\State\Return result
\end{algorithmic}
\end{algorithm}



For a vertex $v$ of a graph $G$, let $w(v)$ denote the sum of the weights assigned by Algorithm \ref{junction_weight_algorithm} to the edges incident with $v$.

\begin{lemma} \label{algorithm_lemma}
Let ${G}$ be a graph with a difference-1 colouring. Then Algorithm \ref{junction_weight_algorithm} assigns ${w(j) = 1}$ for all junctions ${j}$ in ${G}$.
\end{lemma}
\begin{proof}
$G$ has a difference-1 colouring, which means that there is some function $c_1(e)$ that assigns weights of $\pm 1$ to all edges $e$ of $G$ such that for each vertex $v$ of $G$, the sum of weights $w'(v)$ incident with $v$ is $1$. Algorithm \ref{junction_weight_algorithm} implies a function ${c_2: E(G) \rightarrow \mathbb{Z}}$ of integer weight assignments to each edge $e$ in $G$. If $c_2(e) = \pm 1$ to some edge $e$, this means that $e$ must be blue/red in any difference-1 colouring of $G$. If we remove these edges and any vertices which are now isolated, the resulting graph $H$ consists of connected components which are all subgraphs of $Sk(G)$ and are composed of limbs whose ends are both junctions. On each iteration of the algorithm, a junction $j$ is chosen, and any other junction $j'$ encountered while trying to assign weights to $j$ is assigned weights first such that $w(j') = 1$. This process partitions the junctions in the same way as the connected components of $H$. For a component $K$ of $H$, because $w'(v) = w(v) = 1$ in $G$ for any vertex $v \not = j$, the fact that $H$ results from removing only the edges $e$ with $c_2(e) = \pm 1$ from $G$ means that $w'(v) =w(v)$ in $H$ for any vertex $v \not = j$. We also know that $w'(j) = 1$ in $G$. If $(P_1,P_2)$ is the bipartition of $K$ such that $j \in P_1$, because each edge is incident with a vertex in both parts of the bipartition, we have 
\[\sum_{v \in P_1}w'(v) = \sum_{v \in P_2}w'(v) \hspace{1cm} \text{and} \hspace{1cm} \sum_{v \in P_2}w(v) = \sum_{v \in P_1}w(v)\text{.}\]
And because $w'(v) = w(v)$ for all vertices in $Q$, we have
\[\sum_{v \in P_2}w'(v) = \sum_{v \in P_2}w(v)\text{.}\]
Which therefore gives us
\[\sum_{v \in P_1}w'(v) = \sum_{v \in P_1}w(v)\text{.}\]
And because $w'(v) = w(v)$ for all $v \not = j$, this implies that $w'(j)= w(j)$. Therefore the algorithm has assigned $w(j) = 1$ for all junctions $j$ in $G$.
\end{proof}

Note that the weights assigned to each edge incident with a junction can vary depending on the order in which the algorithm deals with the junctions and edges, but the result (whether or not the algorithm concludes that ${G}$ has a difference-1 colouring) is independent of such choices.

\*

It is possible that a graph will have only junctions whose incident weights sum to ${1}$ but have no difference-1 colouring. This is because Algorithm \ref{junction_weight_algorithm} can assign weights to edges ${x(e)}$ that have magnitude greater than ${1}$. We call these weights \emph{surplus weights}, and they (as well as zero weights) arise in the case of edges whose colour may differ in distinct difference-1 colourings of a graph, if any such colourings exist. In order to assign a particular colouring to a graph, these surplus weights must be \emph{redistributed} so that every edge has weight ${1}$ or ${-1}$, while maintaining ${w(j) = 1}$ for each junction ${j}$. If this is possible, we say that the surplus weights are \emph{redistributable}, and we have a criterion for exactly when the surplus weights of any graph are redistributable, which will be outlined in Section \ref{redistrib}.

\begin{example}
Here, some edges of $G$ have been assigned surplus weights. These surplus weights are redistributable, as the sum of the weights at each junction in $Sk(G)$ before and after redistribution remain constant.

\includegraphics[width = \textwidth]{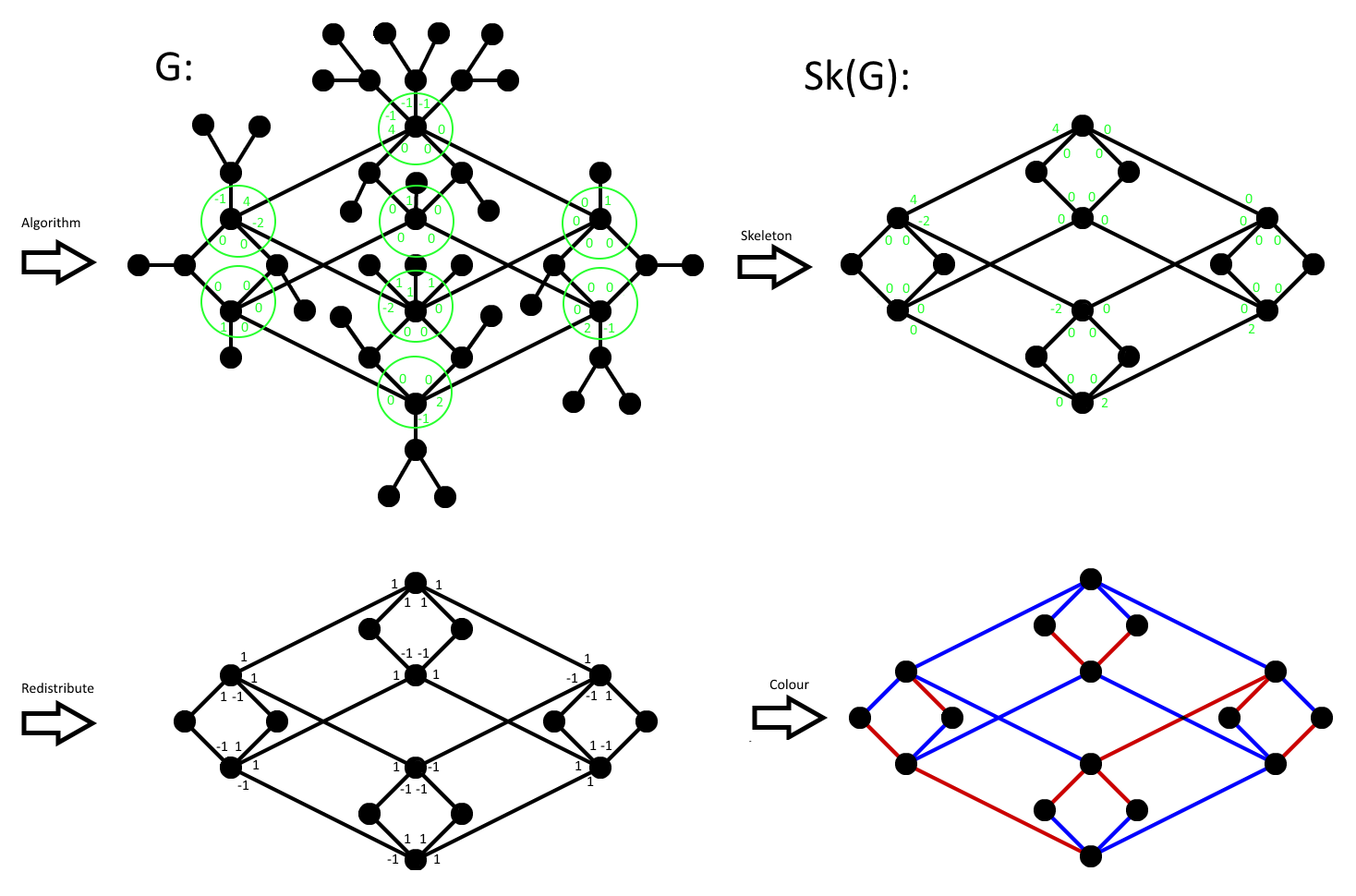}
\end{example}


\subsection{Redistributability}\label{redistrib}
We now consider the situation where Algorithm \ref{junction_weight_algorithm} cannot decide on the existence of a difference-1 colouring for a graph ${G}$. This means that the algorithm assigns weights to the edges of ${G}$ such that ${w(j) = 1}$ for each junction $j$ in $G$, but some of the weights assigned to the edges are surplus weights.

\begin{definition} Let $G$ be a graph for which Algorithm \ref{junction_weight_algorithm} could not determine the existance of a difference-1 colouring, and let $e$ be an edge of $G$ that has been assigned a surplus weight. A \emph{redistribution} of the weight of $e$ is a redefining of the weight of $e$ and of some set $S$ of the other edges of $G$ such that every weight that has been redefined now has a value of $1$ or $-1$, the sum of the weights of all edges incident with each vertex remains $1$, and there is no proper subset of $S$ for which this is possible. We say the surplus weights of a graph are \emph{redistributable} if it is possible to redistribute all surplus weights in $G$ successively. \end{definition}

In order to determine whether or not the surplus weights of a graph are redistributable, we can partition the edge set of the graph into equivalence classes and consider each equivalence class separately, as follows.

\begin{definition} Let ${E'(G)}$ be the set of all edges of ${G}$ that belong to a cycle. For ${e_1, e_2 \in E'(G)}$, ${e_1 \sim e_2}$ if and only if ${e_1}$ and ${e_2}$ occur together  in a cycle of ${G}$. \end{definition}

\begin{lemma} \label{ced_partition} ${\sim}$ is an equivalence relation on $E'(G)$. \end{lemma}
\begin{proof}
It is immediate that ${\sim}$ is reflexive and symmetric. For transitivity, let ${e_1 = u_1v_1}$, ${e_2 = u_2v_2}$, and ${e_3 = u_3v_3}$ be edges of ${G}$ with ${e_1 \sim e_2}$ and ${e_2 \sim e_3}$. Let $C_1$ be a cycle containing both $e_1$ and $e_2$, and $C_2$ be a cycle containing both $e_2$ and $e_3$. If $C_1 = C_2$, then $e_1 \sim e_3$. Assume $C_1 \not = C_2$.

Let ${P_{u_1}}$ be the path in ${C_1}$ starting at ${u_1}$ and ending at a vertex ${u}$ in ${C_2}$, that does not include ${e_1}$ and contains no other vertex in ${C_2}$. Let ${P_{v_1}}$ be the path in ${C_1}$ starting at ${u_1}$ and ending at a vertex ${v}$ in ${C_2}$ that does include ${e_1}$ and contains no other vertex in ${C_2}$. We know that ${u}$ and ${v}$ are distinct, because ${e_2}$ is common to both ${C_1}$ and ${C_2}$. Let ${P_{u_3}}$ be the path in ${C_2}$ starting at ${u}$ and ending at ${u_3}$ which does not include ${e_2}$, and let ${P_{v_3}}$ be the path in ${C_2}$ starting at ${v}$ and ending at ${u_3}$ which does not include ${e_2}$. ${P_{u_1}}$, ${P_{v_1}}$, ${P_{u_3}}$, and ${P_{v_3}}$ form a cycle which contains ${e_1}$ and ${e_3}$.

Therefore ${e_1 \sim e_2, e_2 \sim e_3 \implies e_1 \sim e_3}$; the relation is transitive.
\end{proof}

\n\textsc{Note:} We refer to the equivalence classes of ${\sim}$ as the \emph{common cycle classes} of ${E'(G)}$.

\begin{lemma}\label{common_cycle_class}
Let $G$ be a graph with a difference-$1$ colouring whose edges have been assigned weights by Algorithm \ref{junction_weight_algorithm}. Then all distinct difference-$1$ colourings of $G$ can be obtained by redistributing any surplus and zero weights of $G$ within common cycle classes.
\end{lemma}
\begin{proof}
Define $x: E(G) \to \mathbb{Z}$ to be a function that sends edges of $G$ to the weights that they were assigned by Algorithm \ref{junction_weight_algorithm}, and $x': E(G) \to \{1,-1\}$ to be a redistribution of the surplus weights of $G$. We know that at any junction $j$, the sums of the values assigned by each of these functions to the edges incident with $j$ are $1$, and we also have that $x'(e) = x(e)$ for all $e$ for which $x(e) \in \{1,-1\}$. Let us consider the subgraph $D$ of $G$, consisting of all edges $e$ for which $x(e) \not = x'(e)$, and their incident vertices. Define the $\delta$-weight $\delta : E(D) \to \mathbb{Z}$ by $\delta (e) = x(e)-x'(e)$. Note that for all edges in $D$, the $\delta$-weight is a non-zero integer, and the sum of the $\delta$-weights of edges incident with a vertex in $D$ is $0$. Thus each vertex of $D$ is incident with at least one edge of positive $\delta$-weight and one of negative $\delta$-weight.

\*

Choose any neighbouring vertices $u_0$ and $u_1$ in $D$ and then repeat the following step; on step $i$, choose a vertex $u_{i+1}$ such that ${\delta (u_iu_{i+1})}$ is of the opposite sign to ${\delta (u_{i-1}u_i)}$. This process terminates at some step $j$ when we choose a vertex $u_{j+1} = u_{k}$, where $0 \leq k \leq j$. Because $G$ is bipartite, the cycle $u_ku_{k+1} \dots u_ju_k$ that results from this process is of at least length $4$ and each vertex of the cycle has incident edges with $\delta$-weight of opposite sign. We now define the graph $D'$ to be the graph $D$ with the $\delta$-weight of each edge in this cycle reduced in magnitude by $1$ while retaining its sign, and any edge that now has weight $0$ is deleted. Note that all $\delta$-weights in $D'$ are non-zero integers and the sum of all $\delta$-weights incident with a vertex of $D'$ is $0$, and thus we can repeat this process until we are left with an empty graph. Therefore $D$ can be decomposed into cycles, which implies that surplus weights can be redistributed to other edges in the same common cycle class.
\end{proof}

We now consider when it is possible to redistribute surplus weights in a cactus graph.

\begin{lemma}\label{redistribute_CED} Let ${G}$ be a cactus graph that passes Algorithm \ref{junction_weight_algorithm} and whose edges have been assigned weights. For a cycle $C$ in $G$, let $w_C(v)$ denote the sum of the weights assigned to the edges of $C$ incident with a vertex $v$ in $C$. Then the surplus weights of $G$ are redistributable if and only if $w_C(v) \in \{-2,0,2\}$ for all cycles $C$ in $G$ and vertices $v$ in $C$, and for every path $P = v_1v_2\dots v_k$ in $C$ with $w_C(v_1),w_C(v_k) = \pm 2$ and $w_C(v_i) = 0$ for all $1 < i < k-1$, $k$ is even if $w_C(v_1) = w_C(v_k)$ and odd if $w_C(v_1) = -w_C(v_k)$. \end{lemma}

\begin{proof}
First, assume that the surplus weights of $G$ are redistributable. From Lemma \ref{common_cycle_class}, the surplus weights of $G$ can be redistributed within the common cycle classes of $G$. $G$ is a cactus graph, which means that the common cycle classes of $G$ are the cycles of $G$. Let $C$ be a cycle of $G$. As any redistribution of the surplus weights involves all edges being assigned weights of $\pm 1$, this means that $w_C(v) \in \{-2,0,2\}$ for any vertex $v$ in $C$. If there is a path $P = v_1v_2\dots v_k$ in $C$ such that $w_C(v_1),w_C(v_k) = \pm 2$ and $w_C(v_i) = 0$ for all $1 < i < k-1$, this means that $v_i$ is incident with one edge of each colour in any difference-1 colouring, while $v_1$ and $v_k$ are incident with two edges of the same colour. If $v_1v_2$ is the same colour as $v_{k-1}v_k$, then $k$ must be even and $w_C(v_1) = w_C(v_k)$. If $v_1v_2$ is the opposite colour to $v_{k-1}v_k$, then $k$ must be odd and $w_C(v_1) = -w_C(v_k)$.

\*

Now assume $w_C(v) \in \{-2,0,2\}$ for all cycles $C$ in $G$ and vertices $v$ in $C$, and for every path $P = v_1v_2\dots v_k$ in $C$ with $w_C(v_1),w_C(v_k) = \pm 2$ and $w_C(v_i) = 0$ for all $1 < i < k-1$, $k$ is even if $w_C(v_1) = w_C(v_k)$ and odd if $w_C(v_1) = -w_C(v_k)$. We can redistribute the weights incident with $v_1$ and $v_k$ by assigning a weight of $x = \frac{w_C(v_1)}{2}$ to $v_1v_2$, a weight of $-x$ to $v_2v_3$, \dots, a weight of $(-1)^kx$ to $v_k$. As $k$ is even if $w_C(v_1) = w_C(v_k)$ and odd if $w_C(v_1) = -w_C(v_k)$, this means that $w_C(v_k) = (-1)^kx$. Therefore the surplus weights of $G$ are redistributable.
\end{proof}

To determine in general when surplus weights can be redistributed in a graph $G$  that passes Algorithm \ref{junction_weight_algorithm} and whose edges have been assigned weights, let ${G'}$ be a subgraph of $G$ consisting of all edges from one common cycle class of $G$ and their incident vertices. For each vertex ${v}$ of ${G'}$, let $r(v)$ denote the number of edges of ${G'}$ incident with ${v}$ which are required to be red in any difference-1 colouring of ${G}$ (as determined by Algorithm \ref{junction_weight_algorithm}). Therefore ${r(v)}$ is ${0}$, ${1}$, or ${2}$, if ${v}$ is a twig-type, leaf-type, or triple-type vertex, respectively, and ${r(v) = \frac{1}{2}\big(deg_{G'}(v) - \sum_{uv \in E(G')} x(uv)\big)}$ if ${v}$ is a junction. If we can find a subgraph ${H}$ of ${G'}$ for which ${deg_H(v) = r(v)}$ for each vertex ${v}$ in ${G'}$, then an edge colouring of ${G'}$ where all edges of ${H}$ are coloured red and the remaining edges are coloured blue is consistent with a difference-1 colouring of ${G}$. Therefore, in order to determine if all surplus weights of a graph are redistributable, we need to find a subgraph ${H}$ of $G'$, for each common cycle class $G'$, such that ${deg_H(v) = r(v)}$ for all vertices of ${H}$. The following theorem tells us when this is possible.

\begin{theorem}\label{multimatching} Let ${G}$ be a bipartite graph with bipartition ${(P_1, P_2)}$, and for each vertex ${v}$ of ${G}$, let ${r(v)}$ be an integer value in the range ${0}$ to ${deg_G(v)}$ such that $\sum_{v \in P_1} r(v) \leq \sum_{v \in P_2} r(v)$. Then ${G}$ has a subgraph ${H}$ with ${deg_H(v) = r(v)}$ for every $v \in P_1$ and ${deg_H(v) \leq r(v)}$ for every $v \in P_2$ if and only if every subset ${S}$ of ${P_1}$ in ${G}$ satisfies
\[\sum_{v \in S} r(v) \leq \sum_{n \in \Gamma(S)} \min\{r(n), |\Gamma(n) \cap S|\}\text{,}\]
where ${\Gamma(S)}$ is the set of neighbours of elements of ${S}$ in ${G}$, and ${\Gamma(n)}$ is the set of neighbours of ${n}$.\end{theorem}

Note that the problem of redistributability is precisely the case where $\sum_{v \in P_1} r(v) = \sum_{v \in P_2} r(v)$, as we require all vertices of $v$ in $P_2$ to satisfy $deg_H(v) = r(v)$ as well as for all $v$ in $P_1$. Therefore the surplus weights in a common cycle class $G'$ with bipartition $(P_1,P_2)$ are redistributable if and only if every subset ${S}$ of ${P_1}$ in ${G'}$ satisfies
\[\sum_{v \in S} r(v) \leq \sum_{n \in \Gamma(S)} \min\{r(n), |\Gamma(n) \cap S|\}\text{,}\]
where ${r(v)=0}$ if ${v}$ is a twig-type vertex, ${r(v)=1}$ if ${v}$ is a leaf-type vertex, ${r(v)=2}$ if ${v}$ is a triple-type vertex and 

\[r(v) = \frac{1}{2}\big(deg_{G'}(v) - \sum_{uv \in E(G')} x(uv)\big)\]

if ${v}$ is a junction.

\*

Before we prove this result, we recall some concepts relating to \emph{flow networks}. For our purposes, a flow network is a directed graph $G$ in which every arc is assigned a positive integer weight, called its capacity. The network has exactly one vertex $s$ of indegree zero and positive outdegree, called the source, and exactly one vertex $t$ of outdegree zero and positive indegree, called the sink. A \emph{flow} is an assignment to every arc of a non-negative integer at most equal to its capacity, with the property that for every vertex $v \not \in \{s, t\}$, the total of the flows on arcs directed into $v$ is equal to the total of the flows on arcs directed out of $v$. The total flow is the sum of the flows on all arcs directed out of $s$. The maximum flow of $G$ is the maximum possible total flow over all flows that can be assigned to $G$ (for a fixed assignment of capacity).

\*

A \emph{cut} of a flow network ${G}$ is a partition ${(X,Y)}$ of the vertex set of ${G}$, such that ${s \in X}$ and ${t \in Y}$. The \emph{cut capacity} of a cut ${(X,Y)}$ is the sum of the capacities of all arcs $(u,v)$ from vertex $u$ to vertex $v$ such that ${u \in X, v \in Y}$. We now prove Theorem \ref{multimatching} by interpreting $G$ as a flow network, and applying the \emph{Max-Flow Min-Cut Theorem} \cite{maxminbib}, which states that the maximum flow in a flow network is bounded above by all cut capacities and is equal to the minimum cut capacity.

\begin{proof} 
We construct a directed graph ${G^*}$, with ${V(G^*) = V(G) \cup \{s,t\}}$. The edge set of $G^*$ consists of
\begin{itemize}
\item Arcs $(s,u)$, for all $u \in P_1$;
\item Arcs $(u,v)$, where $u \in P_1$ and $v \in P_2$, for all $uv \in E(G)$;
\item Arcs $(v,t)$, for all $v \in P_2$.
\end{itemize}
We define a flow network structure on $G^*$ by assigning the capacity $r(u)$ to each arc of the form $(s,u)$, $1$ to each arc from $P_1$ to $P_2$, and $r(v)$ to each arc of the form $(v,t)$. The assignment of a flow to $G^*$ determines a subgraph $H$ of $G$ whose edges are the pairs $u,v$ for which the arc $(u,v)$ of $G^*$ has flow $1$.

\*

For a vertex $u \in P_1$, the degree of $u$ in $H$ is equal to the flow assigned to $su$ in $G^*$. Therefore we can find a subgraph $H$ of $G$ where every vertex $v$ of ${H}$ has degree equal to $r(v)$ if and only if we can assign a flow to $G^*$ such that the total flow is equal to $\sum_{v \in P_1} r(v)$. As this desired total flow is equal to the sum of the capacities of all the arcs out of $s$, this is the maximum flow that we could possibly achieve. Therefore $G$ has a subgraph $H$ with the required properties if and only if, for $G^*$,
\[\text{max flow} = \sum_{v \in P_1} r(v) \text{.}\]

From the \emph{Max-Flow Min-Cut Theorem}, this is equivalent to the condition that, for every cut,
\begin{equation} \label{max-min-eq} \text{cut capacity} \geq \sum_{v \in P_1} r(v) \text{.}\end{equation}

To complete the proof of Theorem \ref{multimatching}, we need to show that \eqref{max-min-eq} is equivalent to 
\[\sum_{v \in S} r(v) \leq \sum_{n \in \Gamma(S)} \min\{r(n), |\Gamma(n) \cap S|\} \text{,}\]
for every subset $S$ of $P_1$. 

\*

Let ${S}$ be a subset of ${P_1}$, and choose a cut ${(X,Y)}$ such that ${S = P_1 \cap X}$. We now consider the arcs that contribute to the cut capacity:

\begin{itemize}
\item If a vertex $v$ in $P_1$ is in $Y$, then the arc $(s,v)$ is from $X$ to $Y$ and has capacity ${r(v)}$.
\item If a vertex $v$ in $P_2$ is in $Y$, then any arc $(u,v)$ where $u \in S$ is from $X$ to $Y$. Each such arc has capacity $1$, and therefore the sum of these capacities for a given $v$ is equal to the number of neighbours of $v$ in ${S}$.
\item If a vertex $v$ in ${P_2}$ is in $X$, then the arc $(v,t)$ is from $X$ to $Y$ and has capacity $r(v)$.
\end{itemize}

We therefore have the following:

\[\text{cut capacity} = \sum_{v \in P_1 \cap Y} r(v) + \sum_{v \in P_2} f(v), \;\;\;\; \text{where } f(v) = 
\begin{cases} 
        r(v) & \text{if } v \in P_2 \cap X \\
       |\Gamma(v) \cap S| & \text{if } v \in P_2 \cap Y
   \end{cases}\]

Using this and \eqref{max-min-eq}, our condition for the existence of the required subgraph $H$ becomes
\[\sum_{v \in P_1} r(v) \leq \sum_{v \in P_1 \cap Y} r(v) + \sum_{v \in P_2} f(v)\text{,}\]
or equivalently,
\[\sum_{v \in S} r(v) \leq \sum_{v \in P_2} f(v) \text{,}\]
for every proper subset $S$ of $P_1$, and every choice of cut ${(X,Y)}$ with ${S = P_1 \cap X}$.

\*

The function $f$ depends on the choice of cut, and $\sum_{v \in P_2} f(v)$ is minimized by the cut $(X',Y')$ where
\[X' = S \cup \{v \in P_2 : r(v) \leq |\Gamma(v) \cap S|\}\text{.}\]
Therefore the condition for the existence of $H$ becomes the following. For all $S \subset P_1$, 
\[\sum_{v \in S} r(v) \leq \sum_{n \in \Gamma(S)} \min\{r(n), |\Gamma(n) \cap S|\}\text{.}\]
\end{proof}

As previously shown, Theorem \ref{multimatching} is a special case of the \emph{Max-Flow Min-Cut Theorem}. It is also true that \emph{Hall's Matching Theorem} \cite{hallmatchingbib} is a special case of Theorem \ref{multimatching}. A \emph{matching} in a bipartite graph is a subgraph where every vertex (in the smaller part of the bipartition) has degree ${1}$, and no vertex has degree exceeding $1$.

\begin{theorem}[Hall's Matching Theorem]  Let ${G}$ be a bipartite graph with bipartition ${(P_1, P_2)}$ such that $|P_1| \leq |P_2|$. Then ${G}$ has a matching ${H}$  if and only if every ${S \subset P_1}$ in ${G}$ satisfies
\[|S| \leq |\Gamma(S)| \text{,}\]
where ${\Gamma(S)}$ is the set of neighbours of ${S}$ in ${G}$.\end{theorem}

We note that this is equivalent to Theorem \ref{multimatching}, where ${r(v) = 1}$ for all vertices ${v}$ in the graph. In this case, ${\sum_{v \in S} r(v) = |S|}$, and ${\sum_{n \in \Gamma(S)} \min\{r(n), |\Gamma(n) \cap S|\} = |\Gamma(S)|}$.

\*

We now have the following full set of necessary and sufficient conditions for when a graph has a difference-1 colouring, which follows from Lemma \ref{local-tree-lemma}, Lemma \ref{odd-even-distance-lemma}, Lemma \ref{algorithm_lemma}, Lemma \ref{common_cycle_class}, and Theorem \ref{multimatching}.

\begin{theorem}\label{general-colourability} Let ${G}$ be a bipartite graph containing at least one cycle. Then ${G}$ has a difference-1 colouring if and only if ${G}$ satisfies the following conditions:
\begin{itemize}
\item Each vertex in Sk(G) is either a leaf-type, twig-type, triple-type vertex, or a junction ${j}$ for which the local tree at $j$ in the reduced form of $G$ consists of only leaves or twigs attached to $j$ and which is assigned ${w(j) = 1}$ by Algorithm \ref{junction_weight_algorithm};
\item For each path ${P = v_1v_2 \dots v_k}$ in ${Sk(G)}$ where ${v_1}$ and ${v_k}$ are both non-leaf type vertices, and ${v_2, \dots, v_{k-1}}$ are all leaf-type vertices, ${P}$ has odd (even) length if ${v_1}$ and ${v_k}$ are the same (opposite) type;
\item Surplus weights in each common cycle class of ${G}$ are redistributable. 
\end{itemize}
\end{theorem}

Recall from Theorem \ref{ced_difference-1_is_asbg} that if a cactus graph ${G}$ has a difference-1 colouring ${c}$, then $G^c$ is configurable. Therefore the Theorem \ref{general-colourability} determines when a cactus graph is ASBG-colourable. The problem of determining necessary and sufficient conditions for configurability of difference-1 colourings for wider classes of graphs is an ongoing topic of investigation.


\section{Uniqueness and Difference-k Colourings}

In this last section, we examine what it means for one graph to have more than one distinct difference-1 colouring, and how these difference-1 colourings can differ from one another. We also explore generalising the concept of a difference-1 colouring to a \emph{difference-k} colouring.

\subsection{Uniqueness of Difference-1 Colourings}

\begin{definition} We say that a graph ${G}$ has a \emph{unique colouring} if there is only one difference-1 colouring of ${G}$. \end{definition}

It is possible that a graph can have multiple difference-1 colourings, some of which are configurable and some of which are not.

\begin{example}
The following graph has two distinct difference-1 colourings, only one of which is configurable:

\includegraphics[width = \textwidth]{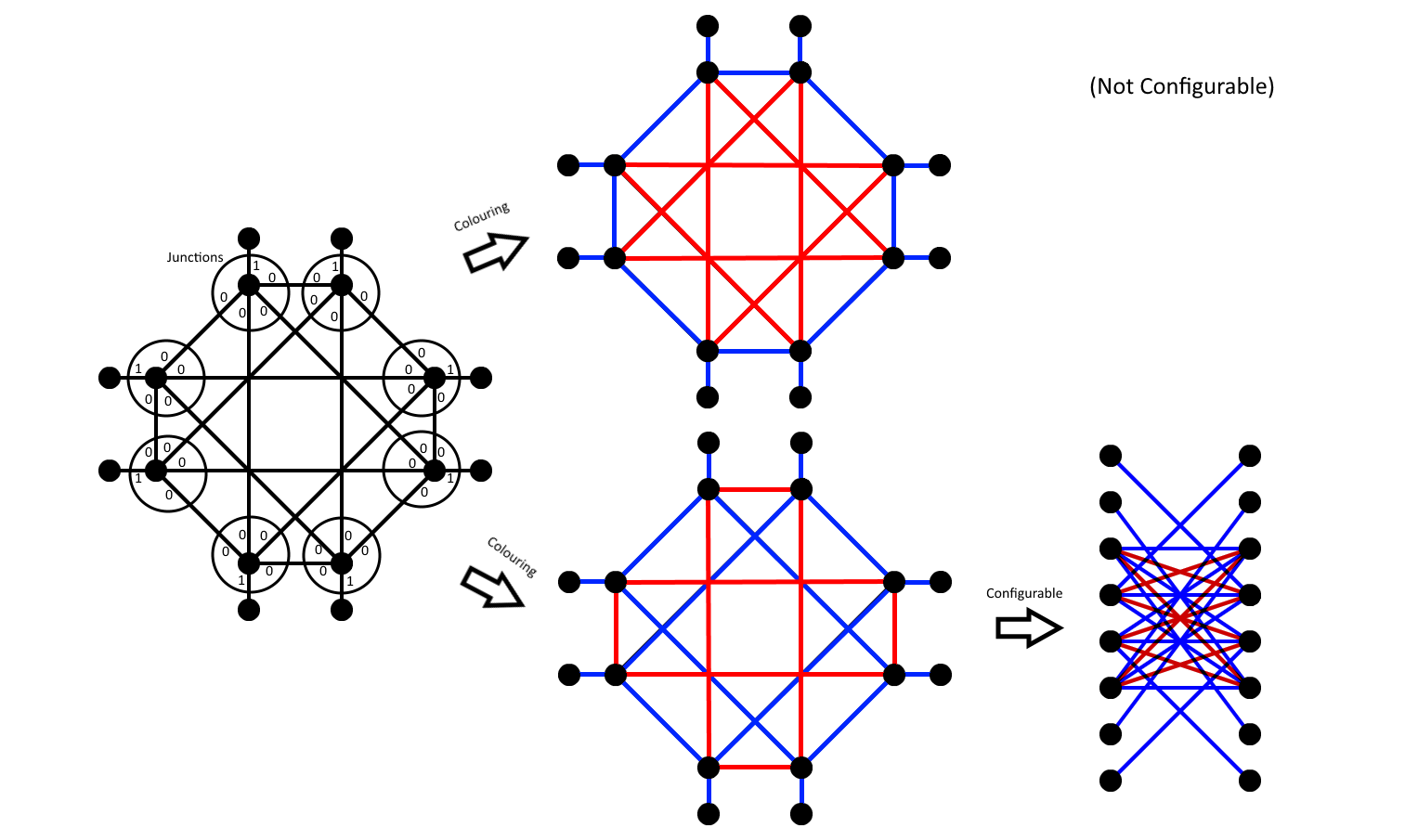}
\end{example}

\begin{definition} An \emph{alternating cycle} is a cycle in a graph which has been coloured such that each vertex of the cycle is incident with two edges in the cycle of opposite colour. \end{definition}

\begin{definition} The \emph{rotation} of an alternating cycle ${C^c}$ is the coloured cycle ${C^d}$ where every edge of ${C^d}$ is the opposite colour of the corresponding edge in ${C^c}$. \end{definition}

\begin{theorem} Any distinct colourings of a graph ${G}$ differ only by edge-disjoint alternating cycle rotations. \end{theorem}
\begin{proof}
Let ${G}$ be a graph with two distinct colourings ${c}$ and ${d}$, and let ${H}$ be the subgraph of ${G^c}$ containing only the edges which differ in colour from those in ${G^d}$ and only vertices which are incident with these edges. Since both $c$ and $d$ are difference-1 colourings, every vertex of $H$ is incident with the same number of red and blue edges, and so every vertex of $H$ has even degree. We may start at any edge of $H$, and construct a trail consisting of edges of alternating colours until a repeated vertex $v$ is encountered for the first time. This completes an alternating cycle from $v$ to itself, since every cycle in $G$ has even length. Removal of this cycle from $H$ results in another graph in which every vertex is incident with an equal number of blue and red edges. So we can remove alternating cycles until no edges remain, and therefore two distinct colourings of a graph ${G}$ differ only by alternating cycle rotations.
\end{proof}

\begin{corollary} A difference-1 colouring of a tree ${T}$ is the unique ASBG-colouring of ${T}$. \end{corollary}

\subsection{Difference-k Colourings}

The definition and exploration of difference-1 colourings leads very naturally to the more general definition of a \emph{difference-k} colouring.

\begin{definition} \label{dif-k_def} A 2-edge colouring $c$ of a graph ${G}$ is a \emph{difference-$k$ colouring}, $k \in \mathbb{N}_0$, of $G$ if ${G}$ is bipartite and $G^c$ satisfies $deg^B(v) - deg^R(v) = k$ at every vertex $v$. \end{definition}

We propose the problem of determining when a graph admits a difference-k colouring, and note some preliminary observations in the bipartite case. The following is a characterisation of bipartite graphs with a difference-0 colouring.

\begin{theorem} A bipartite graph $G$ has a difference-0 colouring if and only if every vertex of $G$ has even degree. \end{theorem}

\begin{proof}
First, assume that $G$ has a difference-0 colouring. Then $deg^B(v) - deg^R(v) = 0 \implies deg^B(v) = deg^R(v) \implies deg(v) = 2deg^B(v)$. So every vertex of $G$ has even degree.

\*

Now assume that every vertex of $G$ has even degree. As $G$ is \emph{Eulerian}, the edges of $G$ can be decomposed into edge-disjoint cycles. We now colour the edges of $G$ so that each cycle is an alternating cycle. Therefore $deg^B(v) = deg^R(v)$ at each vertex $v$. So $G$ has a difference-0 colouring.
\end{proof}

We can use Theorem \ref{multimatching} to give a general characterisation for all bipartite graphs with difference-k colourings, for a given $k$ as follows.

\begin{theorem}\label{dif-k} A bipartite graph $G$ with bipartition $(P_1,P_2)$ has a difference-k colouring if and only if every $S \subset P_1$ satisfies
\[\sum_{v \in S} \frac{deg(v)-k}{2} \leq \sum_{n \in \Gamma(S)} \min\Big\{\frac{deg(n)-k}{2}, |\Gamma(n) \cap S|\Big\}\text{,}\]
where ${\Gamma(S)}$ is the set of neighbours of elements of ${S}$ in ${G}$, and ${\Gamma(n)}$ is the set of neighbours of ${n}$.\end{theorem}

\begin{proof}
A bipartite graph $G$ having a difference-k colouring is equivalent to $G$ having a subgraph $H$ where every vertex $v$ of $H$ satisfies $deg_H(v) = \frac{deg_G(v)-k}{2}$. This is because if we find such a subgraph $H$, the colouring $c$ of the edges of $G$ such that all edges of $H^c$ are red and all edges of $(G \setminus H)^c$ are blue is a difference-k colouring of $G$. From Theorem \ref{multimatching}, we know that this is equivalent to every strict subset $S$ of one bipartition of $G$ satisying
\[\sum_{v \in S} \frac{deg(v)-k}{2} \leq \sum_{n \in \Gamma(S)} \min\Big\{\frac{deg(n)-k}{2}, |\Gamma(n) \cap S|\Big\}\text{.}\]
\end{proof}

While Theorem \ref{dif-k} arises directly from Theorem \ref{multimatching}, and can be applied to the case of difference-1 colourings, its practical utility is limited. Algorithm \ref{junction_weight_algorithm} and partitioning the edge set into common cycle classes simplifies the analysis considerably in the case $k=1$. At present, we have no version of these approaches for difference-k colourings in general. The problem of determining when a non-bipartite graph admits a difference-k colouring remains largely unexplored, even in the case $k=1$.

\end{document}